\documentclass[12pt]{article}

\usepackage{tikz}
\usetikzlibrary{intersections,calc,arrows.meta}
\usetikzlibrary{arrows.meta,decorations.pathmorphing,backgrounds,positioning,fit,petri,through}
\usetikzlibrary{datavisualization}
\usetikzlibrary{decorations.pathreplacing}
\usetikzlibrary{angles,quotes}

\usepackage{amssymb} % \mathbb
\usepackage{amsthm}
\usepackage{amsmath}
\usepackage{cite}
\usepackage{latexsym}
\usepackage{mathrsfs}
\usepackage{enumerate}
\usepackage{graphicx}
\usepackage{fancybox}
\usepackage{color}
\usepackage{multicol, multienum}
\usepackage{bm}
\usepackage{mathtools}
\mathtoolsset{showonlyrefs=true}

\allowdisplaybreaks

\newtheorem{theorem}{Theorem}[section]
\newtheorem{lemma}[theorem]{Lemma}

\newtheorem{corollary}[theorem]{Corollary}

\theoremstyle{definition}
\newtheorem{definition}[theorem]{Definition}

\definecolor{mgre}{cmyk}{0.92,0.00,0.59,0.25}

\definecolor{PineGreen}{cmyk}{0.92,0.00,0.59,0.25}
\definecolor{ForestGreen}{cmyk}{0.91,0.00,0.88,0.12}
\definecolor{RawSienna}{cmyk}{0.00,0.72,1.00,0.45}
\definecolor{Mulbery}{cmyk}{0.34,0.90,0.00,0.02}
\definecolor{Sepia}{cmyk}{0.00,0.83,1.00,0.70}
\definecolor{Mahogany}{cmyk}{0.00,0.85,0.87,0.35}

\def\bE{{\mathbb{E}}}

\def\bM{{\mathbb{M}}}

\def\bP{{\mathbb{P}}}

\def\bR{{\mathbb{R}}}

\def\bfL{{\mathbf{L}}}
\def\bfM{{\mathbf{M}}}

\def\cB{{\mathcal{B}}}

\def\cD{{\mathcal{D}}}
\def\cE{{\mathcal{E}}}
\def\cF{{\mathcal{F}}}

\def\cH{{\mathcal{H}}}

\def\cK{{\mathcal{K}}}

\def\cS{{\mathcal{S}}}

\def\1{{\mathbf{1}}}

\newtheorem{thm}{Theorem}[section]

\newtheorem{prop}[thm]{Proposition}
\newtheorem{lem}[thm]{Lemma}
\newtheorem{cor}[thm]{Corollary}

\theoremstyle{definition}
\newtheorem{defi}[thm]{Definition}
\newtheorem{prob}{}[section]
\newtheorem{rem}[thm]{Remark}
\newtheorem{ex}[thm]{Example}
\newtheorem{pro}{Problem}[section]
\newtheorem{ass}[thm]{Assumption}

\newcommand{\bd}{\begin{defi}}
\newcommand{\ed}{\end{defi}}

\newcommand{\bpro}{\begin{pro}}
\newcommand{\epro}{\end{pro}}

\newcommand{\bec}{\begin{cases}}
\newcommand{\eec}{\end{cases}}

\newcommand{\bpr}{\begin{prob}}
\newcommand{\epr}{\end{prob}}

\newcommand{\bt}{\begin{thm}}
\newcommand{\et}{\end{thm}}

\newcommand{\ba}{\begin{ass}}
\newcommand{\ea}{\end{ass}}

\newcommand{\br}{\begin{rem}}
\newcommand{\er}{\end{rem}}

\newcommand{\bpm}{\begin{pmatrix}}
\newcommand{\epm}{\end{pmatrix}}

\newcommand{\be}{\begin{ex}}
\newcommand{\ee}{\end{ex}}

\newcommand{\bp}{\begin{prop}}
\newcommand{\ep}{\end{prop}}

\newcommand{\bl}{\begin{lem}}
\newcommand{\el}{\end{lem}}

\newcommand{\bc}{\begin{cor}}
\newcommand{\ec}{\end{cor}}

\newcommand{\bq}{\begin{que}}
\newcommand{\eq}{\end{que}}

\newcommand{\beqn}{\begin{eqnarray*}}
\newcommand{\eeqn}{\end{eqnarray*}}

\newcommand{\beqnn}{\begin{eqnarray}}
\newcommand{\eeqnn}{\end{eqnarray}}

\newcommand{\bequ}{\begin{equation}}
\newcommand{\eequ}{\end{equation}}

\newcommand{\benu}{\begin{enumerate}}
\newcommand{\eenu}{\end{enumerate}}

\newcommand{\barr}{\begin{array}{rcl}}
\newcommand{\ear}{\end{array}}

\newcommand{\la}{\label}

\newcommand{\rea}{\mathrm{Re}\ }

\title{A note on geometric $\alpha$-stable processes and the existence of
  ground states for associated Schr\"odinger operators}
\author{
\textsc{Kaneharu Tsuchida} %names of authors
}
\date{\today}
 
\begin{document}

\maketitle

\begin{abstract}
  In this paper, we establish the existence of transition density
  for geometric $\alpha$-stable processes by using the property of
  self-decomposability--a fundamental concept
  in the theory of L\'evy processes.
  In contrast to traditional and analytic methods that often rely on the
  $L^{1}$-integrability of the characteristic function, our approach is
  purely probabilistic and focuses on the structural regularity of
  the L\'evy measure.
  As an application, we prove the existence of ground states for Schr\"odinger
  operators associated with recurrent geometric stable processes. 
  \end{abstract}
  \bigskip

   \medskip
  
  \noindent
  {\bf AMS 2020 \textit{Mathematics Subject Classification}}.
  Primary 60J45; Secondary 60J76, 31C25.

  \section{Introduction}
  {\v{S}}iki{\'{c}}, Song and Vondra\v{c}ek \cite{SikicSongVondracek2006}
  developed the potential
  theory of geometric stable processes using the theory of
  subordinate Brownian motions. In particular, they demonstrated the  
  asymptotic behaviors of the L\'evy density and Green function
  and established the Harnack inequality.
  In this note, in relation to their results, we consider the existence of
  transition density and ground states 
  of the Schr\"odinger operator generated by the recurrent
  geometric stable process.
  
  Let $\mathbf{M} = (\mathbb{P}_x, X_t)$ be
  the geometric $\alpha$-stable process on $\mathbb{R}^d$.
  The generator of $\mathbf{M}$ is the pseudo-differential operator
  $\mathcal{H} = -\log(1 + (-\Delta)^{\alpha/2})$ for $0 < \alpha \le 2$.
  The process $\mathbf{M}$ is a purely discontinuous Lévy process.
  In the specific case where $\alpha = 2$,
  $\mathbf{M}$ is commonly referred to as the ``variance gamma process."
  Let $P_t$ and $p_t(dx)$ denote the transition semigroup
  and the transition function of $\mathbf{M}$, respectively.
  For any non-negative Borel function $f$ and $x \in \mathbb{R}^d$,
  we have
  \begin{align}
    \label{eq:6}
    P_t f(x) = \mathbb{E}_{x}[f(X_t)] = \int_{\mathbb{R}^d} f(x+y) p_t(dy).
  \end{align}
  The symbol of $\mathcal{H}$ is given by $\psi(\xi) = \log(1 + |\xi|^\alpha)$.
  According to the Chung-Fuchs criterion for symbols,
  it is easily seen that $\mathbf{M}$ is recurrent
  if $d \le \alpha$ and transient otherwise.

  The transition density $p_t(x)$ of the geometric stable process
  has been the subject of extensive study from various
  analytical perspectives \cite{Linnik1963, Kotz2001Laplace,
    BogdanEtAl2009LNM, GrzywnyRyznarTrojan2019IMRN}.
  The distribution of $X_{1}$ coincides with the symmetric Linnik distribution,
  whose explicit density representations involving series expansions
  are detailed in Linnik \cite{Linnik1963} and Kotz et al. \cite{Kotz2001Laplace}.
  
  More recently, Bogdan et al. \cite{BogdanGrzywnyRyznar2014JFA}
  established sharp two-sided estimates for the transition density
  of a broad class of isotropic unimodal Lévy processes.
  However, they explicitly excluded limiting cases
  where the scaling index (upper Matuszewska index) is $0$,
  leaving the precise estimation for the geometric stable process
  as an open problem.
  In fact, they remarked that such processes require
  a specialized approach distinct from their general scaling methods.
  We note that specific asymptotic estimates were
  derived in \cite[Section 5.3.4]{BogdanEtAl2009LNM} relying on
  the specific subordination structure of the process.
  For instance, small-time estimates near the origin
  exhibit a time-dependent singularity \cite[Theorem 5.52]{BogdanEtAl2009LNM},
  while large-space asymptotics at $t=1$ reveal a transition
  from polynomial decay for $\alpha \in (0,2)$ to exponential decay
  for $\alpha=2$ \cite[Proposition 5.51]{BogdanEtAl2009LNM}.
  While these works provide valuable pointwise information,
  deriving global regularity properties from them requires piecing
  together these complex local and tail behaviors.
  
  In contrast, we focus on the \textit{self-decomposability} of the process.
  This structural approach provides an alternative and direct verification
  of the existence and regularity of the density for all $t > 0$,
  bypassing the need for detailed pointwise estimates or subordination arguments.
  Crucially, it immediately implies the strong Feller property,
  providing a solid and self-contained analytical framework
  for our subsequent analysis of the ground state.

  Let us clarify the technical difficulties underlying the existence of the density.
  Recall that the characteristic function of $X_t$ is
  \begin{align}
    \label{eq:3}
    \Phi(t,\xi) = e^{-t\psi(\xi)} = \frac{1}{(1 + |\xi|^\alpha)^t}.
  \end{align}
  If $\Phi(t,\xi)$ were $L^1$-integrable for all $t > 0$,
  the density $p_t(x)$ could be obtained via the standard
  Fourier inversion formula:
  \begin{align} \label{eq:2}
    p_t(x) = \frac{1}{(2\pi)^d} \int_{\mathbb{R}^d}
    e^{-i \langle x,\xi \rangle} \Phi(t,\xi) d\xi.
  \end{align}
  However, $\Phi(t, \xi)$ fails to be integrable whenever $t \le d/\alpha$.
  Consequently, for small $t$, the existence of a transition density
  is not guaranteed by standard Fourier analytic methods.
  Furthermore, the geometric stable process does not satisfy
  the Hartman-Wintner condition \cite{Tucker:1965, Knopova:Schilling:2013}
  (a sufficient condition for the existence of a smooth density)
  due to the logarithmic growth of its characteristic exponent at infinity.

  Thus, a different strategy is required.
  Instead of relying on Fourier analysis,
  we utilize the structural property of self-decomposability to
  establish the absolute continuity of $\mathbf{M}$ for all $t > 0$.
  Intuitively, a distribution is said to be self-decomposable
  if it can be decomposed into a scaled version
  of itself plus an independent remainder (see \cite{Sato1999Levy} for details).
  This property guarantees the absolute continuity
  of the distribution whenever the Lévy measure is infinite.
  By verifying the self-decomposability of geometric $\alpha$-stable processes,
  we provide a self-contained proof of the existence of the transition density.
  Complementary to the quantitative estimates provided in existing literature,
  our focus on self-decomposability highlights the intrinsic probabilistic structure
  of the geometric stable process.
  Beyond merely establishing the existence of the transition density,
  this approach offers a clear conceptual link between
  the process's algebraic decomposition and its analytic regularity,
  which is particularly advantageous for establishing the strong Feller property
  without involved integral computations.
  Since the absolute continuity of transition probabilities
  and the strong Feller property are equivalent for Lévy processes
  (cf. \cite{Hawkes1979Potential}),
  our structural verification of the former directly establishes the latter.
  This analytical foundation facilitates our subsequent investigation
  into the spectral properties of the process.
  
  As an application of the existence of the transition density,
  we establish the existence of a ground state for the Schr\"odinger operator
  generated by the recurrent geometric $\alpha$-stable process.
  As stated above, $\bfM$ is reccurent if $d \le \alpha$. 
  In contrast to the transient case, the Green function is not
  well-defined in the recurrent case, making the analysis of the
  ground state significantly more delicate.
  The case of symmetric stable processes was studied in \cite{Takeda2016},
  and that of relativistic stable processes was investigated in \cite{Tsuchida2020}.
  
  Further details are provided below.  
  Let $\mu = \mu^{+} - \mu^{-}$ be a signed Kato measure, where we assume
  that both $\mu^{+}$ and $\mu^{-}$ are non-trivial (non-zero) measures. 
  Let $\bfM^{\mu^+}$ denotes the process $\bfM$ killed by $\mu^+$, 
  and let $(\cF_{e}^{\mu^+}, \cE^{\mu^+})$ be the associated extended Dirichlet space, 
  where $\cE^{\mu^+}(u,u) = \cE(u,u) + \int_{\bR^{d}} u^{2} d\mu^{+}$.
  Then, we consider the following quantity:
  \begin{align}
  \label{eq:27}
    \lambda := \inf \left\{ \cE^{\mu^+}(u,u) :
    u \in \cF_{e}^{\mu^+}, \int_{\bR^{d}}^{ } u^{2} d\mu^{-} = 1\right\}.
  \end{align}
  Here $(\cF_{e}^{\mu^+}, \cE)$ denotes the extended Dirichlet space 
  associated with $\bfM^{\mu^+}$ and $\bfM^{\mu^+}$ is the killed 
  process of $\bfM$ by $\mu^+$. 
  If there exists a minimizer $h$ such that $h$ attains the infimum in \eqref{eq:27},
  then $\lambda$ represents the principal eigenvalue of
  the geometric $\alpha$-stable process killed by $\mu^{+}$ and time-changed
  by $\mu^{-}$. We refer to such a funtion $h$ as a {\it $\lambda$-ground state}.
  The $\lambda$-ground state $h$ serves as the ground state associated with
  $-\cH + \mu^{+} - \lambda \mu^{-}$ because it satisfies
  $\cE^{\mu^{+} - \lambda \mu^{-}}(h,h) = ((-\cH + \mu^{+} - \lambda \mu^{-})h,h) = 0$. 
  In the case where $\lambda = 1$, this is precisely the {\it ground state} of
  the Schr\"odinger operator $-\cH + \mu$. 
  To prove its existence, 
  we employ the Class (T) method (see Takeda \cite{Takeda2019}).
  This method establishes that the semigroup $\{P_{t}\}$ is compact,
  and as a corollary, guarantees the compactness of the embedding
  from the Dirichlet space into $L^{2}$-space. 
  One of important conditions for the Class (T) approach
  is the strong Feller property.
  As noted earlier, for L\'evy processes,
  the existence of a transition density (absolute continuity) is equivalent to
  the strong Feller property (cf.~\cite{Hawkes1979Potential}).
  Consequently, having established the density,
  we are able to prove the existence of the ground state
  under appropriate assumptions on $\mu^{-}$.

  The rest of this paper is organized as follows.
  In Section 2, we summarize the basic notions
  and known results concerning asymptotic behaviors 
  of L\'evy densities of geometric stable processes.
  Section 3 is devoted to proving the existence of the transition density.
  Finally, in Section 4, we establish the existence of
  the ground state in the recurrent case as an application of the result
  established in the previous section.

  \section{Preliminaries}
  In this section, we provide the necessary background
  on Fourier transforms of probability measures, L\'evy processes
  and the definition of geometric stable processes.
  
  Let $\mathcal{P}(\mathbb{R}^{d})$ be the set of probability measures on $\mathbb{R}^{d}$.
  For a measure $\mu \in \mathcal{P}(\mathbb{R}^{d})$, its Fourier transform
  $\widehat{\mu}(\xi)$ is defined by
  \begin{align}
  \label{eq:F-1}
    \widehat{\mu}(\xi) &= \int_{\mathbb{R}^{d}}^{ } e^{i \left\langle \xi,x \right\rangle} \mu(dx), \quad \xi \in \mathbb{R}^{d}, 
  \end{align}
  where $\left\langle \cdot, \cdot \right\rangle$ denotes the standard inner product on
  $\bR^{d}$. 
  
  Let $\bfL = (\bP, L_{t})$ be a L\'evy process starting at $0$ defined on a probability space
  $(\Omega,\mathcal{F},\bP)$. 
  It is well known that the distribution of $L_{t}$ is uniquely determined by its symbol
  (or characteristic exponent) $\psi(\xi)$ via the relation,
  \begin{align}
  \label{eq:symbol}
    \widehat{\mu_{t}}(\xi) = \bE\left[ e^{i \left\langle \xi, L_{t} \right\rangle} \right] = e^{-t \psi(\xi)},\quad
    \xi \in \bR^{d},\ t \ge 0, 
  \end{align}
  where $\bE$ denotes the expectation of $\bP$.
  That is, $e^{-t \psi(\xi)}$ is nothing but the characteristic function of $L_{t}$.

  For any $x \in \bR^{d}$, we define the process $(\bP_{x}, X_{t})$ by
  \begin{align}
  \label{eq:def}
  X_{t} = x + L_{t}, \quad t \ge 0.
  \end{align}
  The characteristic function of $X_{t}$ under the law $\bP_{x}$ is then given by
  \begin{align}
  \label{eq:cha}
    \bE_{x} \left[ e^{i \left\langle \xi, X_{t} \right\rangle} \right] 
    = e^{i \left\langle \xi, x \right\rangle} e^{-t\psi(\xi)},
  \end{align}
  where $\bE_{x}$ denotes the expectation associated with $\bP_{x}$.

  In this paper, we specifically consider the case where the symbol is
  given by $\psi(\xi) = \log{(1 + |\xi|^{\alpha})}$ with $\alpha \in (0,2]$. 
  The corresponding L\'evy process $\bfM$ is referred to as a
  {\it geometric $\alpha$-stable process}. 
  In the sequel, we focus only on this characteristic function. 
  Let $p_{t}(dy)$ be the transition probability measure and
  $\left\{ P_{t}, t > 0\right\}$ be the trasition semigroup of $\mathbf{M}$.
  Since $\mathbf{M}$ is a L\'evy process, the semigroup acts on $f \in \cB_{b}(\bR^{d})$ as 
  \begin{align}
  \label{eq:1}
    P_{t} f(x) = \bE_{x}\left[ f(X_{t}) \right] = \int_{\bR^{d}}^{ } f(x+y) p_{t}(dy).
  \end{align}
  Now, we review the L\'evy-Khintchine representation for the symbol
  $\psi(\xi)$. Since the symbol $\psi(\xi) = \log{(1 + |\xi|^{\alpha})}$ is symmetric
  and satisfies $\psi(0) = 0$, the diffusion, drift and killing terms vanish.
  Hence, $\psi(\xi)$ admits the following integral representation:
  \begin{align}
  \label{eq:7}
    \psi(\xi) = \int_{\bR^{d}}^{ } \left( 1 - \cos{\left\langle \xi, y \right\rangle} \right)
    J(dy).
  \end{align}
  The measure $J$ is referred to the L\'evy measure, it satisfies
  $\int_{\bR^{d} \setminus \{0\}} (1 \wedge |y|^{2}) J(dy) < \infty$.

  \if0
  Although we do not use the explicit formula in our proof,
  it is worth noting that Sikic et al. \cite{SikicSongVondracek2006}
  established that the Lévy measure $J$ admits a density.
  We recall their result as follows.

  Since the geometric $\alpha$-stable process is a subordinate
  symmetric $\alpha$-stable process with the gamma subordinator,
  we know that the L\'evy measure is absolutely continuous with
  respect to the Lebesgue measure by \cite[Theorem 30.1]{Sato1999Levy}.
  Although we do not use this explicit formula in our proof,
  it is worth noting that Sikic et al. established
  in \cite{SikicSongVondracek2006} the following results
  regarding the Lévy measure $J$ associated with $\bfM$."
  \fi
  
  Although we do not use the explicit formula in our proof,
  it is worth noting that the Lévy measure $J$ admits a density,
  whose explicit form is derived in \cite{SikicSongVondracek2006}.
  We state their result as follows:

  \begin{theorem}[{\bf \cite[Theorem 3.4--3.6]{SikicSongVondracek2006}}]
    \la{thm:SSV}
    Let $J$ be the L\'evy measure associated with the geometric $\alpha$-stable
    process $\bfM$. For every $\alpha \in (0,2]$, there exists a density function $j(x)$
    of $J$ with respect to the Lebesgue measure. Moreover, the following holds:
    \begin{enumerate}[{\rm (1)}]
    \item For every $\alpha \in (0,2]$, we have
      \begin{align}
      \label{eq:8}
        j(x) \sim \frac{\alpha \Gamma(d/2)}{2|x|^{d}}, \quad |x| \to 0.
      \end{align}
    \item For every $\alpha \in (0,2)$, we have
      \begin{align}
      \label{eq:9}
        j(x) \sim \frac{\alpha}{2^{\alpha + 1} \pi^{d/2}}
        \frac{\Gamma \left( \frac{d+\alpha}{2} \right)}
        {\Gamma \left( 1 - \frac{\alpha}{2} \right)} |x|^{-d-\alpha}, \quad |x| \to \infty.
      \end{align}
    \item For $\alpha = 2$, we have
      \begin{align}
      \label{eq:10}
        j(x) \sim 2^{-d/2} \pi^{-\frac{d-1}{2}} \frac{e^{-|x|}}{|x|^{\frac{d+1}{2}}}, \quad |x| \to \infty.
      \end{align}
    \end{enumerate}
  \end{theorem}

  Since $\psi(\xi)/|\xi|^{\alpha} \to 1$ as $|\xi| \to 0$,
  it follows from the Chung-Fuchs criterion that $\mathbf{M}$ is recurrent
  if and only if $d \le \alpha$. This property is consistent with that of
  the standard $\alpha$-stable process.

  \section{Existence of the transition density}
  While the transition density of the geometric stable process has been extensively studied, 
  we here present a direct derivation based on the structural properties of L\'evy processes.
  A direct application of standard criteria is not straightforward; for instance,
  the symbol $\psi(\xi) = \log{(1 + |\xi|^{\alpha})}$ does not satisfy the Hartman-Wintner condition:
  \begin{align}
    \label{eq:5}
    \lim_{|\xi| \to \infty} \frac{\rea \psi(\xi)}{\log{(1 + |\xi|)}} = \infty,
  \end{align}
  which typically guarantees the existence of smooth densities.
  Instead of employing complex analytical bounds, 
  we use the structural property of \textit{self-decomposability}.
  We begin by defining the following notion.

  \begin{definition}
    \la{def:SD}
    Let $\mu \in \mathcal{P}(\bR^{d})$ be an infinitely divisible distribution on $\bR^{d}$.
    The measure $\mu$ is said to be self-decomposable if
    for any $b > 1$, there exists $\rho_{b} \in \mathcal{P}(\bR^{d})$ such that
    \begin{align}
    \label{eq:4}
    \widehat{\mu}(\xi) = \widehat{\mu}(b^{-1} \xi) \widehat{\rho_{b}} (\xi).
    \end{align}
  \end{definition}

  We state the following two lemmas.

  \begin{lemma}[{\cite[Theorem 27.13]{Sato1999Levy}}]
    \la{lem:SDtoAC}
    Any nondegenerate self-decomposable distribution on $\bR^{d}$ is
    absolutely continuous with respect to the Lebesgue measure.
  \end{lemma}

  By virtue of Lemma \ref{lem:SDtoAC}, we establish the existence of
  the transition density kernel of $\bfM$.
  For this purpose, we prepare the following lemma.

  \begin{lemma}[{\cite[Theorem 15.10]{Sato1999Levy}}]
    \la{lem:eqSD}
    Let $S = \{\xi \in \bR^{d} : |\xi| = 1\}$. 
    Let $\mu$ an infinitely divisible distribution on $\bR^{d}$ with
    the L\'evy measure $J(dx)$. Then $\mu$ is self-decomposable
    if and only if
    \begin{align}
    \label{eq:11}
      J(B) = \int_{S^{d-1}} \lambda(d\theta) \int_{0}^{\infty} 1_{B}(r\theta)
      \frac{k_{\theta}(r)}{r} dr,  \quad \text{for } B \in \mathcal{B}(\bR^{d} \setminus \{0\})
    \end{align}
    with a finite measure $\lambda$ on $S^{d-1}$ and a nonnegative
    function $k_{\theta}(r)$ measurable in $\theta \in S^{d-1}$ and decreasing in $r > 0$. 
  \end{lemma}

  Under two lemmas, we show that the existence of transition density.

  \begin{theorem}
    \la{thm:pdf}
    Let $d \ge 1$ and $\alpha \in (0,2]$. Then the geometric $\alpha$-stable
    process $\bM$ has the transition density with respect to the Lebesgue
    measure. 
  \end{theorem}
  \begin{proof}
    Let $\mu_{t}$ be the distribution of $X_{t}$ under the probability measure $\bP_{x}$. 
    First, consider the case at time $t = 1$. Note that the L\'evy process $\bfM$ is
    obtained by
    subordinating the symmetric $\alpha$-stable process with a gamma process.
    We remark that its L\'evy density is given by
    \begin{align}
    \label{eq:LevyDense}
      j(x) = \int_0^{ \infty} q_{s}(x) \frac{e^{-s}}{s} ds, \quad x \in \bR^{d} \setminus \{0\},
    \end{align}
    where $q_{s}(x)$ is the transition density of the standard symmetric
    $\alpha$-stable process. One can also verify equation \eqref{eq:LevyDense}
    through a direct calculation.

    Using the scaling property
    $q_{s}(r\xi) = s^{-d/\alpha}q_{1}(s^{-1/\alpha}r\theta)$ and
    the polar coordinate transformation $x = r \theta$ (with $r > 0$ and
    $\theta \in S^{d-1}$), we obain
    \begin{align}
    \label{eq:14}
      J(B) &:= \int_{B}^{ } j(x) dx = \int_{B}^{ } \left(\int_0^{ \infty} q_{s}(x)
             \frac{e^{-s}}{s} ds\right) dx \\
           &= \int_{\bR^{d}} \int_0^{ \infty} 1_{B}(x) s^{-d/\alpha} q_{1}(s^{-1/\alpha} x) \frac{e^{-s}}{s} ds \\
           &= \int_{S^{d-1}} \sigma(d\theta) \int_0^{ \infty}
             1_{B}(r\theta) \left(\int_{0}^{\infty } s^{-d/\alpha}
             q_{1} (s^{-1/\alpha} r\theta) \frac{e^{-s}}{s}ds\right) r^{d-1} dr,
    \end{align}
    where $\sigma$ is the surface measure on $S^{d-1}$. 
    Comparing this with equation \eqref{eq:11}, we see 
    \begin{align}
    \label{eq:12}
      k_{\theta}(r) = r^{d} \int_0^{ \infty} s^{-d/\alpha}
      q_{1}(s^{-1/\alpha} r\theta) \frac{e^{-s}}{s} ds.
    \end{align}
    Here, we regard a finite measure $\lambda$
    in lemma \ref{lem:eqSD} as $\sigma$.
    Hence it is enough to prove that the function $k_{\theta}(r)$ is
    decreasing on $r > 0$.
    By setting $u = s^{-1/\alpha} r$, we have $\frac{ds}{s} = - \alpha \frac{du}{u}$. So
    \begin{align}
    \label{eq:13}
      k_{\theta}(r) &= r^{d} \int_{\infty}^{ 0} \left( \frac{u}{r} \right)^{d} q_{1} (u \theta)
                      e^{-\left( \frac{r}{u} \right)^{\alpha}} \left( -\alpha \frac{du}{u} \right) \\
                    &= \alpha \int_0^{ \infty} u^{d-1} q_{1}(u\theta) e^{-\left( \frac{r}{u} \right)^{\alpha}} du.
    \end{align}
    Therefore $k_{\theta}(r)$ is decreasing for $r > 0$.
    In the sequel, we refer to $k_{\theta}(r)$ as the $k$-function of the distribution $\mu_{1}$. 
    It then follows from Lemma \ref{lem:eqSD}
    that the distribution $\mu_{1}$ is self-decomposable.
    Moreover, according to the existence of the Lévy density given in Theorem \ref{thm:SSV},
    the Lévy measure possesses a strictly positive density on $\mathbb{R}^d \setminus \{0\}$;
    hence, the L\'evy measure is non-degenerate. 
    By Lemma \ref{lem:SDtoAC}, we see that
    $\mu_{1}$ is absolutely continuous with respect to the Lebesgue measure.

    For any $t > 0$, the L\'evy measure of $X_{t}$ is given by $t J(dx)$; hence
    the $k$-function associated with $X_{t}$ is $t k_{\theta}(r)$.
    Since $k_{\theta}(r)$ is decreasing for $r > 0$,
    the monotonicity of $t k_{\theta}(r)$ is preserved,
    which implies that the distribution of $X_t$ is self-decomposable for all $t > 0$.
    
    Consequently, by Lemma \ref{lem:SDtoAC},
    the distribution $\mu_{t}$ of $X_t$ is absolutely continuous
    with respect to the Lebesgue measure for all $t > 0$. 
  \end{proof}

  \section{Application to the existence of ground states}
  In this section, we only consider the case that $\alpha \ge d = 1$.
  Then the geometric $\alpha$-stable process $\bfM$ is recurrent. 
  Let $(\cE,\cF)$ be the Dirichlet form generated by $\bfM$.
  By Example \cite[Example 1.4.1]{FOT}, we see
  \begin{align}
  \label{eq:15}
    \begin{dcases}
      \cE(u,v) = \int_{\bR^{d}} \widehat{u}(x) \bar{\widehat{v}}(x) (\log{(1 + |x|^{\alpha})}) dx,\\
      \cF = \left\{ u \in L^{2}(\bR^{d}) : \int_{\bR^{d}} |\widehat{u}(x)|^{2} (\log{(1 + |x|^{\alpha})}) dx < \infty \right\},
    \end{dcases}
  \end{align}
  where $\widehat{u}$ denotes the Fourier transform of $u$ defined by
  \begin{align}
  \label{eq:16}
    \widehat{u}(x) = (2\pi)^{-d/2} \int_{\bR^{d}}^{ } e^{i \left\langle x,y \right\rangle} u(y) dy, \quad x \in \bR^{d}.
  \end{align}
  The Dirichlet form $(\mathcal{E}, \mathcal{F})$ on
  $L^2(\mathbb{R}^d; m)$ associated with $\bfM$
  admits the Beurling-Deny representation. 
  Since $\bfM$ is a conservative L\'evy process with purely jump path,
  the diffusion and killing parts of the representation vanish.
  Specifically, for any $u, v \in \mathcal{F}$,
  the form is expressed as:
  \begin{align}
  \label{eq:25}
    \mathcal{E}(u, v) = \iint_{\mathbb{R}^d \times \mathbb{R}^{d} \setminus d}
    (u(x) - u(y))(v(x) - v(y)) J_1(dx, dy),
  \end{align}
  where $J_1(dx, dy)$ is the jumping measure and
  $d = \{(x,x) \in \bR^{d} \times \bR^{d}: x \in \bR^{d}\}$. 
  In the case of $\bfM$ on $\mathbb{R}^d$,
  the jumping measure is given by $J_1(dx, dy) = \frac{1}{2} j(x-y) dx dy$,
  where $j(r)$ is the L\'evy density of the process (see Theorem \ref{thm:SSV}). 
  Moreover, note that it is also known that the Dirichlet form $(\cE,\cF)$ is regular.
  
  Let $\cH$ be the generator of $\bfM$, that is,
  \begin{align}
  \label{eq:20}
  \cH = -\log{(1 + (-\Delta)^{\alpha/2})}.
  \end{align}

  Let $\cS_{1}$ be a class of smooth measaures in the strict sense
  with respect to the geometric
  $\alpha$-stable process $\bfM$.
  Note that the absolute continuity of $\bfM$ ensures that
  such measures and their corresponding
  positive continuous additive functionals (PCAFs) can be well-defined
  in the strict sense (see, e.g., \cite[Section 5.1]{FOT} for details). 
  Under the framework of Dirichlet forms,
  the Revuz correspondence provides a unique one-to-one mapping
  between the set of all smooth measures and the set of
  PCAFs (up to equivalence). Now, let us define the Kato class. 

  \begin{definition}
    \la{def:Kato}
    For a measure $\mu \in \cS_{1}$, let $\{A_{t}^{\mu}\}$ be the PCAF
    associated with $\mu$ via the Revuz correspondence.
    If $A_{t}^{\mu}$ satisfies
    \begin{align}
    \label{eq:17}
      \lim_{t \downarrow 0} \sup_{x \in \bR^{d}} \mathbb{E}_{x}\left[ A_{t}^{\mu} \right] = 0, 
    \end{align}
    the measure $\mu$ is said to be in Kato class $\mathcal{K}$. 
  \end{definition}

  Using the PCAF $A_{t}^{\mu}$ associated with $\mu \in \cK$,
  we can define the killed process $\bfM^{\mu}$. The transition semigroup of this killed process,
  denoted by $\{P_{t}^{\mu}\}_{t \ge 0}$, is given by the Feynman-Kac formula:
  \begin{align}
  \label{eq:18}
    P_{t}^{\mu} f(x) = \mathbb{E}_{x}\left[ e^{-A_{t}^{\mu}} f(X_{t}) \right].
  \end{align}
  Then the process $\bfM^{\mu}$ is transient.
  Since the killed semigroup $P_t^{\mu}$ is subordinate to the original semigroup
  $P_t$ (i.e., $P_t^{\mu} f \le P_t f$ for $f \ge 0$),
  the absolute continuity of $P_t$ with respect to
  the Lebesgue measure immediately implies the absolute continuity of $P_t^{\mu}$.
  
  For a measure $\mu \in \cK$,
  let $G^{\mu}(x,y)$ be the minimal Green function of $\bfM^{\mu}$.
  \begin{definition}
    \la{def:GT}
    Let $\nu \in \cK$. The measure $\eta$ is said to be in a Green-tight Kato class
    $\cK_{\infty}^{\nu}$ if $\eta$ satisfies
    \begin{align}
    \label{eq:19}
      \lim_{r \to \infty} \sup_{x \in \bR^{d}} \int_{|y| \ge r}^{ } G^{\nu}(x,y) \eta(dy) = 0.
    \end{align}
  \end{definition}

  Let $\mu = \mu^{+} - \mu^{-}$ be a signed measure which satisfies
  $\mu^{+} \in \cK$ and $\mu^{-} \in \cK_{\infty}^{\mu^{+}}$.
  Here we assume 
  that both $\mu^{+}$ and $\mu^{-}$ are non-trivial (non-zero) measures.
  Let $\bfM^{\mu^{+}}$ be the killed process of $\bfM$ by $\mu^{+}$.
  Putting $\cF^{\mu^{+}} = \cF \cap L^{2}(\mu^{+})$
  and $\cE^{\mu^{+}}(u,u) = \cE(u,u) + \int_{\mathbb{R}^{d}}^{ } u^{2} \mu^{+}$,
  then $(\cE^{\mu^+}, \cF^{\mu^+})$ is the Dirichlet form associated with $\bfM^{\mu^{+}}$.
  Let $\cF_e^{\mu^+}$ be the extended Dirichlet space of $\cF^{\mu^+}$. 

  Consider the Schr\"odinger type operator $\cH^{\mu} = \cH + \mu$ and
  variational problem:
  \begin{align}
  \label{eq:24}
    \lambda = \inf \left\{ \cE^{\mu^{+}}(u,u) :
    u \in \cF_e^{\mu^{+}}, \int_{\bR^{d}}^{ } u^{2} d\mu^{-} = 1\right\}.
  \end{align}
  As mentioned in the introduction, 
  a function attaining the lower bound in \eqref{eq:24},
  if it exists, is called a {\it $\lambda$-ground state}.
  Recall that this is a ground state of the Schr\"odinger operator 
  $-\cH + \mu^{+} - \lambda \mu^{-}$
  In particular, when $\lambda = 1$,
  such a function is referred to as a {\it ground state} of the operator $-\cH + \mu$. 
 
 To establish the $\lambda$-ground state, we use the Takeda's class (T) method
 (see \cite{Takeda2019}).
 In what follows, we provide a sketch of the approach for our setting. 
 To discuss the general theory, let us once again denote a general
 symmetric Markov process by $\bfM = (\bP_{x}, X_{t})$, and $P_{t}$
 (resp. $R_{\alpha}$) denotes its transition semigroup (resp. $\alpha$-resolvent).
 We denote the Dirichlet form associated with $\bfM$ by $(\cE,\cF)$,
 and assume that $(\cE,\cF)$ is regular. 
 Let $m$ denotes the Lebesgue measure on $\bR^{d}$. 
 In \cite{Takeda2019}, the author imposed the following three conditions
 on the Markov process $\bfM$:
 \begin{enumerate}[(I)]
 \item (Irreducibility) If a Borel set $A$ is $P_{t}$-invariant, i.e.,
   $\int_{A}^{ } P_{t} 1_{A^{c}} dm = 0$ for any $t > 0$, then $A$ satisfies either
   $m(A) = 0$ or $m(A^{c}) = 0$.
 \item (Resolvent Strong Feller Property) $R_{\alpha} (\cB_{b}) \subset C_{b},\
   \alpha > 0$.
 \item (Green-Tightness Property) For any $\varepsilon > 0$, there exists
   a compact set $K$ such that $\sup_{x \in E} R_{1} 1_{K^{c}}(x) \le \varepsilon$. 
 \end{enumerate}
 If $\bfM$ satisfies these three conditions, $\bfM$ is called ``Class (T)''
 Then the following result is obtained.

 \begin{theorem}[{\cite[Theorem 4.3 and Corollary 4.5]{Takeda2019}}]
   \la{thm:T}
   If $\bfM$ is in Class (T), then the semigroup $P_{t}$ is a compact operator
   on $L^{2}(E,m)$.
   Moreover, $(\cD(\cE), \cE_{1})$ is compactly embedded in $L^{2}(E,m)$.
\end{theorem}
We refer to the approach of obtaining such compactness in this manner as
the ``Class (T) method''.
Additionally, the following is obtained as a corollary of Theorem \ref{thm:T}. 
 
\begin{corollary}
  \la{cor:T}
  If a symmetric Markov process satisfying (I) and (II) is transient,
  then $(\cF_{e}, \cE)$ is compactly embedded in $L^{2}(E,\mu)$ for each
  $\mu \in \cK_{\infty}^{0}$. Here $\cK_{\infty}^{0}$ is the class of measures
  satisfying \eqref{eq:19} in Definition \ref{def:GT} with $\nu = 0$. 
\end{corollary}

In the sequel, we use the Class (T) method
to obtain a ground state $h$ %for \eqref{eq:24}
(i.e., $h$ achieves the infimum in \eqref{eq:24}). 
To this end, we consider a symmetric Markov process as $\bfM^{\mu^{+}}$ and
take $\mu$ to be $\mu^{-}$ in Corollary \ref{cor:T}. 

  \begin{lemma}
    \la{lem:Irr}
    The killed geometric $\alpha$-stable process $\bfM^{\mu^{+}}$ is irreducible, 
    that is,
    if a Borel set $A$ is $P_{t}^{\mu}$-invariant,
    i.e., $\int_{A}^{ } P_{t}^{\mu^{+}} 1_{A^{c}}(x) dx = 0$
    for any $t > 0$, then either $A$ or $A^{c}$ is a null set with
    resepct to the Lebesgue measure. 
  \end{lemma}
  \begin{proof}
    Suppose the process is not irreducible.
    Then there exists a non-trivial $P_{t}^{\mu}$-invariant set $A$,
    such that $m(A) > 0$ and $m(A^{c}) > 0$.
    By \cite[Theorem 1.6.1]{FOT}, we see that
    $u1_{A} \in \cF$ for any $u \in \cF$ with $u > 0$ on $\bR^{d}$ and
    the following decomposition holds:
    \begin{align}
    \label{eq:21}
    \cE^{\mu}(u,u) = \cE^{\mu}(1_{A}u, 1_{A} u) + \cE^{\mu}(1_{A^{c}}u, 1_{A^{c}}u). 
    \end{align}
    Note that the killing part of the cross-term vanishes
    because the supports are disjoint:
    \begin{align}
    \label{eq:26}
    \int 1_{A}u(x) 1_{A^{c}}u(x) \mu(dx) = 0.
    \end{align}
    Thus, the equation \eqref{eq:21} implies that
    the cross-term $\cE(1_{A}u, 1_{A^{c}}u)$ of the original part must be zero.

    On the other hand,
    a direct calculation yields
    \begin{align}
      \label{eq:22} \cE(1_{A}u, 1_{A^{c}}u) &= \iint_{\bR^{d}\times\bR^{d}} (1_{A}u(x) - 1_{A}u(y))(1_{A^{c}}u(x) - 1_{A^{c}}u(y)) J_1(dx,dy) \notag \\
                                            &= - \int_{A} \int_{A^{c}} u(x) u(y) j(x-y)dxdy.
    \end{align}
    
    Since $u(x)$ and $j(x)$ are strictly positive on
    $\bR^{d}$ by \eqref{eq:LevyDense}, the cross term 
    is strictly negative unless $m(A)=0$ or $m(A^c)=0$.
    This contradiction proves the irreducibility.
  \end{proof}

  For $\alpha > 0$, let $R_{\alpha}^{\mu^{+}}$ be the
  $\alpha$-resolvent operator of $\{P_{t}^{\mu}\}$, that is,
  \begin{align}
  \label{eq:23}
    R_{\alpha}^{\mu^{+}} f(x) &= \int_0^{ \infty} e^{-\alpha t} P_{t}^{\mu^{+}} f(x) dt,
                      \quad \text{for } f \in \cB_{b}.
  \end{align}
  %where this integral is the Bochner's sense. 

  The next result is one of applications for the existence of transition density. 
 \begin{lemma}
   \la{lem:SF}
   The process $\bfM^{\mu^{+}}$ has the strong Feller property,
   that is, $P_{t}^{\mu^{+}}(\cB_{b}) \subset C_{b}$. 
 \end{lemma}
 \begin{proof}
   First, let us recall that any L\'evy process generates a convolution semigroup.
   By virtue of the smoothing effect of convolution,
   it is well known that the absolute continuity of the L\'evy process
   is equivalent to the strong Feller property of its associated semigroup
   (see also \cite[Lemma 2.1]{Hawkes1979Potential}).
   Hence the geometric $\alpha$-stable process $\bfM$
   has the strong Feller property. 
   
   Next, to prove the strong Feller property of $\bfM^{\mu^{+}}$,
   using the same method of 
   Chung and Zhao \cite{ChungZhao1995}
   (see also \cite[Proposition 4.4]{Tsuchida2008Potential}),
   we obtain the strong Feller property of $\bfM^{\mu^{+}}$. 
 \end{proof}

 Finally, we provide the conclusion of this section.

 \begin{theorem}
   \la{thm:main}
   Let $\mu$ be a signed measure with $\mu^{+} \in \cK$ and
   $\mu^{-} \in \cK_{\infty}^{\mu^{+}}$. If both $\mu^{+}$ and $\mu^{-}$ are non-trivial
   measures, there exists a positive continuous and bounded
   $\lambda$-ground state $h$ 
   in \eqref{eq:24}.
   Furthermore, such a function $h$ is uniquely determined. 
 \end{theorem}
 \begin{proof}
   We apply Corollary \ref{cor:T} to $\bfM^{\mu^{+}}$.
   It is clear that $\bfM^{\mu^{+}}$ is transient, and note that
$\cK_{\infty}^{\mu^{+}}$ corresponds to $\cK_{\infty}^{0}$ in Corollary \ref{cor:T}. 
   The condition (I) follows from Lemma \ref{lem:Irr}.
   In view of Lemma \ref{lem:SF} and the fact that  the strong Feller property
   entails the resolvent strong Feller property 
  (i.e., $R_{\alpha}^{\mu^+}(\cB_b) \subset C_b$), the condition (II) is satisfied. 
   Therefore, the existence of the $\lambda$-ground state is guaranteed
   by Corollary \ref{cor:T}. 
   The uniqueness, continuity and boundedness follow from
   \cite[Theorem 3.3]{Takeda2019} and \cite[Theorem 5.4]{Takeda2019}.
 \end{proof}

 \section*{Acknowledgments}
 The author would like to express his sincere gratitude to
 Professor Masayoshi Takeda and Professor Yuichi Shiozawa for
 their kind comments and suggestions, which helped improve this paper
 in many ways.

\textsc{Department of Mathematics, National Defense Academy, Yokosuka,
239--8686, Japan}

{\it Email address:} tsuchida@nda.ac.jp

\end{document}